\documentclass[leqno]{amsart}
\usepackage[margin=1in]{geometry}
\usepackage{graphicx,amssymb, amsfonts, amsmath, amsthm, mathtools, hyperref, tikz-cd, xcolor, leftindex, enumerate} 

\title{Dilogarithm identities of higher degree\\ and cluster $a$-variable periodicity}

\author{Zachary Nash}
\address[Z. Nash]{\href{mailto:zach.n.nash@gmail.com}{\textup{\texttt{zach.n.nash@gmail.com}}}, Math Academy, Pasadena High School, Pasadena, CA 91107 USA}
\author{Dylan Rupel}
\address[D. Rupel]{\href{mailto:rupel.dylan@pusd.us}{\textup{\texttt{rupel.dylan@pusd.us}}}, Math Academy, Pasadena Unified School District, Pasadena, CA 91109 USA}

\newcommand{\QQ}{\mathbb{Q}}
\newcommand{\ZZ}{\mathbb{Z}}
\newcommand{\RR}{\mathbb{R}}
\newcommand{\PP}{\mathbb{P}}
\newcommand{\TT}{\mathbb{T}}
\renewcommand{\aa}{\boldsymbol{a}}
\newcommand{\cc}{\boldsymbol{c}}
\newcommand{\dd}{\boldsymbol{d}}
\newcommand{\qq}{\boldsymbol{q}}
\newcommand{\yy}{\boldsymbol{y}}
\newcommand{\zz}{\boldsymbol{z}}
\newcommand{\LL}{\tilde{L}}
\newcommand{\Li}{\operatorname{Li}}
\DeclareMathOperator{\Trop}{Trop}
\DeclareMathOperator{\diag}{diag}
\newtheorem{theorem}{Theorem}[section]
\newtheorem{lemma}[theorem]{Lemma}
\newtheorem{cor}[theorem]{Corollary}
\newtheorem{prop}[theorem]{Proposition}
\newtheorem{definition}[theorem]{Definition}
\newtheorem{example}[theorem]{Example}
\newtheorem{remark}[theorem]{Remark}
\numberwithin{equation}{section}

\begin{document}

\begin{abstract}
  We use the periodicities of cluster groupoid mutations established by Li and the second author to prove that the dilogarithm identities of higher degree obtained by Nakanishi follow from the classical dilogarithm identities associated to a periodicity of cluster mutations when the exchange matrix is full rank.
\end{abstract}

\maketitle

\section{Introduction}
  
  The Rogers and Euler dilogarithms are important functions in mathematics, satisfying a variety of equations known as \emph{dilogarithm identities} \cite{zagier2007dilogarithm}.
  These functions are deeply related to cluster algebra mutations \cite{nakanishi2011periodicities}.
  Cluster algebras are commutative rings which are obtained by setting some initial data, known as the \emph{seed}, and then repeatedly \emph{mutating} that data to obtain new seeds \cite{fomin2002cluster}.
  The seeds obtained through mutations are organized in the \emph{exchange graph} and the cycles in this graph are called \emph{periodicities of seed mutations}.
  Any periodicity of seed mutations gives rise to a dilogarithm identity \cite{kashaev2011classical,nakanishi2011periodicities}.

  Cluster algebra mutations were generalized in \cite{chekhov2014teichmuller}.
  This generalization is obtained through the insertion of additional data, known as the \emph{z-variables}, which modify the standard binomial mutation rule into a polynomial one.
  Motivated by the connection between cluster algebras and dilogarithm identities, Nakanishi defined the \emph{dilogarithms of higher degree} in order to establish an analogous connection between periodicities of generalized cluster mutations and dilogarithm identities \cite{nakanishi2018rogers}.
  By factoring the associated polynomial, this function may be split into a sum of (the analytic continuations of) ordinary dilogarithms, reducing dilogarithm identities of higher degree to just the ordinary case.
  However, to do this properly, one must be incredibly careful with choosing the path of analytic continuation; our approach avoids this complication by keeping the generalized dilogarithm as one package.

  In \cite{li2020symplectic}, Li and Rupel introduced symplectic groupoids integrating a compatible Poisson structure on a cluster variety.
  Here all choices of the $z$-variables are considered together so that they become the coordinates on an affine space.
  In order to obtain a symplectic groupoid, this necessitates the introduction of corresponding fiber coordinates, designated as $a$-variables in \cite{li2020symplectic}.
  Mutations for these groupoids were obtained through the Hamiltonian perspective, first introduced in \cite{fock702397quantum}.
  In this way, the mutations of the \(a\)-variables were described by a Hamiltonian flow associated with an Euler dilogarithm of higher degree.
  In \cite{li2020symplectic}, it was established that the periodicity of \(a\)-variables follows directly from a periodicity of the generalized cluster mutations.
  This periodicity states that a certain sum of integral expressions is zero.
 
  The \(z\)-variables appear in the associated dilogarithm identities, but prior research on the dilogarithms of higher degree left those variables fixed \cite{nakanishi2018rogers}.
  The major goal of this paper is to investigate and quantify how changing the initial \(z\)-variables affects the dilogarithm expressions appearing in the identities.
  We show that a modified form of the $a$-variable periodicity formulas controls the dependence of the dilogarithm identities on the $z$-variables and that this modified sum of integral expressions being zero follows from the $a$-variable periodicity when the exchange matrix is full rank.
  This shows that the dilogarithm identity of higher degree is independent of the $z$-variables in this case.
  By setting the $z$-variables to zero, we arrive at a classical dilogarithm identity associated to a periodicity of ordinary cluster mutations which thus implies the dilogarithm identities of higher degree.

\section{Generalized cluster algebras and their groupoids} 
  \label{sec:generalized cluster algebras}
  
  Ordinary cluster algebras were introduced in 2002 by Fomin and Zelevinsky \cite{fomin2002cluster}. 
  Chekhov and Shapiro  \cite{chekhov2014teichmuller} later generalized this concept by introducing certain \(p\)-coefficients, yielding generalized cluster algebras.  
  In \cite{iwaki2016exact}, \cite{nakanishi2015structure}, and \cite{nakanishi2015companion} it was shown that \(p\)-coefficients satisfying a normalization condition and a power condition can be repackaged into \(z\)-variables, which greatly helps in simplifying the mutation rule.
  These mutation rules were extended in \cite{li2020symplectic} to produce a groupoid structure.
  
  For the purposes of this paper, it will mainly be useful to consider not the generalized cluster algebras as a whole, but rather their underlying structure: their seeds and mutations.

  \subsection{Generalized Cluster Seeds}
    \begin{definition}
      A semifield is a set \(\PP\) with an additive binary operation \(\oplus\) and a multiplicative binary operation \(\cdot\), such that:
      \begin{enumerate}[(1)]
        \item \((\PP,\cdot)\) is an abelian group free of torsion;
        \item the additive operation \(\oplus\) is associative and commutative;
        \item the multiplicative operation distributes over addition, i.e. 
        \[x(y\oplus z) = xy\oplus xz \quad \text{for all \(x,y,z \in \PP\).}\]
      \end{enumerate}
      Let $\PP_0:=\PP\cup\{0\}$, where $0$ is an identity element for $\oplus$.
    \end{definition}
  
    \begin{example}
      The set $\RR_{>0}$ is naturally a semifield under ordinary addition and multiplication of real numbers.
    \end{example}
  
    \begin{example}
      \label{ex:universal semifield}
      Let \(\yy = (y_i)_{i=1}^n\) be an \(n\)-tuple of formal commuting variables.  
      Then the \emph{universal semifield} \(\QQ_+(\yy)\) is defined to be the set of all subtraction-free rational functions \(P(\yy)/Q(\yy)\), i.e. where \(P(\yy)\) and \(Q(\yy)\) are non-zero polynomials in \(\yy\) with non-negative integer coefficients.
      Its ``universality" comes from the fact that, given any semifield \(\PP\), and any elements \(t_1,\cdots,t_n \in \PP\), there exists a unique semifield homomorphism \(\varphi: \QQ_+(\yy) \to \PP\) such that \(\varphi(y_i) = t_i\) for all \(1 \leq i \leq n\) \cite[Lemma 2.1.6]{berenstein1996parametrizations}.
      In other words, every equality that holds in \(\QQ_+(\yy)\) still holds when you substitute in elements of another semifield.
    \end{example}
  
    \begin{example}
      Let \(\yy = (y_i)_{i=1}^n\) be an \(n\)-tuple of formal commuting variables.  
      Then we define the \emph{tropical semifield} of \(\yy\) as
      \[\Trop(\yy) \coloneq \left\{\prod\limits_{i=1}^n y_i^{a_i} | \ a_i \in \ZZ\right\},\]
      where multiplication is defined as usual, and addition is defined by
      \[\prod\limits_{i=1}^n y_i^{a_i} \oplus \prod\limits_{i=1}^n y_i^{b_i} \coloneq \prod\limits_{i=1}^n y_i^{\min(a_i,b_i)}.\]
    \end{example}

    Following \cite{nakanishi2018rogers}, we construct the seed of a generalized cluster algebra as follows.  
    Let \(n\) be a strictly positive integer, called the \emph{rank}, and let \(\PP\) be a semifield, called the \emph{coefficient semifield}.  
    We also fix an \(n\)-tuple of positive integers \(\dd = (d_i)_{i=1}^{n}\) called the \emph{mutation degrees}.
  
    \begin{definition}
      A \emph{$y$-seed}, or simply a \emph{seed}, in a generalized cluster algebra over $\PP$ is defined to be a triple \((B,\yy,\zz)\), where:
      \begin{itemize}
        \item \(B = (b_{ij})_{i,j=1}^n\) is an \(n \times n\) skew-symmetrizable integer matrix, called the \emph{exchange matrix}, with a fixed choice of diagonal matrix \(R\) with positive integer entries, called the \emph{skew-symmetrizer}, such that \(RB\) is skew-symmetric, i.e. \((RB)^T = -RB\);
        \item \(\yy = (y_i)_{i=1}^n\) is an \(n\)-tuple of elements of \(\PP\), called the \emph{$y$-variables};
        \item \(\zz = \{z_{i,s} | \ 1 \leq i \leq n; 1 \leq s \leq d_i-1\}\) is a collection of elements from \(\PP_0\), called the \emph{$z$-variables}.
      \end{itemize}
      When $d_i=1$ for all $i$, there are no $z$-variables and we write the seed simply as $(B,\yy)$.
    \end{definition}
  
    For \(1 \leq k \leq n\) and \(\alpha \in \PP\), we define 
    \[\quad P_k(\alpha) \coloneq 1 \oplus \bigoplus\limits_{s=1}^{d_k-1} z_{k,s} \alpha^s \oplus \alpha^{d_k}.\]
    For $a\in\RR$, set \([a]_+ \coloneq \max(a,0)\).
  
    \begin{definition}
      \label{def:seed mutations}
      For any integer \(1 \leq k \leq n\) and choice of sign \(\varepsilon \in \{1,-1\}\), the \emph{mutation} of \((B,\yy,\zz)\) in direction \(k\) is a new seed \((B',\yy',\zz')\) defined as follows:
      \begin{align*}
        b_{ij}' &\coloneq \begin{cases} -b_{ij} & \text{if $i=k$ or $j=k$;}\\ b_{ij}+[-\varepsilon b_{ik} d_k]_{+} b_{kj}+b_{ik}[\varepsilon d_k b_{kj}]_+ & \text{if $i,j \neq k$;}\end{cases}\\
        y_i' &\coloneq \begin{cases} y_k^{-1} & \text{if $i=k$;}\\ y_i y_k^{[\varepsilon d_k b_{ki}]_+} \big(P^\circ_k(y_k^\varepsilon)\big)^{-b_{ki}}  \hspace{0.6in} & \text{if $i \neq k$;} \end{cases}\\
        z_{i,s}' &\coloneq \begin{cases} z_{k,d_k-s} \hspace{1.6in} & \text{if $i=k$;}\\ z_{i,s} & \text{if $i \neq k$;}\end{cases}
      \end{align*}
      where $P^\circ_k(\alpha) \coloneq P_k(\alpha^\varepsilon) \alpha^{(\frac{1-\varepsilon}{2})d_k}$.
    \end{definition}
    An easy check shows that $B'$ is still skew-symmetrizable with the same skew-symmetrizer $R$ and that the formulas do not actually depend on the choice of sign \(\varepsilon\).

    Let $\TT_n$ denote the $n$-regular tree with root vertex $t_0$ and edges adjacent to each vertex labeled by $1,\ldots,n$.
    Fix an initial seed $(B_{t_0},\yy_{t_0},\zz_{t_0})$ and assign seeds $(B_t,\yy_t,\zz_t)$ for $t\in\TT_n$ so that $(B_t,\yy_t,\zz_t)$ and $(B_{t'},\yy_{t'},\zz_{t'})$ are related by mutation in direction $k$ when $t$ and $t'$ are joined by an edge labeled by $k$ in $\TT_n$.
    This produces the \emph{exchange pattern} associated to the seed $(B_{t_0},\yy_{t_0},\zz_{t_0})$.
  
    As a special case, we consider when \(\PP\) is the universal semifield \(\PP = \QQ_+(\yy,\zz)\) and take an initial seed where \(\yy_{t_0} = \yy\) and \(\zz_{t_0} = \zz\) are formal commuting variables.
    This allows to define the \emph{tropical sign} \(\varepsilon_{j;t}\) for $t\in\TT_n$ and $1\le j\le n$ as follows.
    Let \(\varphi_{\Trop}: \QQ_+(\yy,\zz) \to \Trop(\yy,\zz)\) be the unique semifield homomorphism such that \(\varphi_{\Trop}(y_i) = y_i\) and \(\varphi_{\Trop}(z_{i,s}) = z_{i,s}\) for \(1 \leq i \leq n\) and \(1 \leq s \leq d_i-1\).
    Then it is known \cite[Lemma 3.6]{nakanishi2015structure} that 
    \[
      \varphi_{\Trop}(y_{j;t}) 
      =
      \prod\limits_{i=1}^n y_i^{c_{ij;t}}, \ (c_{ij;t})_{i=1}^n \in \ZZ^n
    \]
    for all \(1 \leq j \leq n\) and $t\in\TT_n$.
    In other words, there is no dependency on the \(z\)-variables.
    This gives us integer matrices \(C_t = (c_{ij;t})_{i=1,j=1}^{n}\) associated with $t\in\TT_n$, called the \emph{C-matrices}.
    Alternatively, these \(C\)-matrices can be described by the following recursion \cite[Proposition 3.8]{nakanishi2015structure}:
    \[
      \label{eq:C-matrix 1}
      c_{ij;t_0} = \delta_{ij}
    \]
    and, when $t,t'\in\TT_n$ are joined by an edge labeled $k$, the $C$-matrices $C_t$ and $C_{t'}$ are related by
    \[
      c_{ij;t'} = \begin{cases} -c_{ik;t} & \text{if $j=k$;} \\ c_{ij;t}+[-c_{ik;t} d_k]_+ b_{kj;t}+c_{ik;t}[d_k b_{kj;t}]_+ & \text{if $j \neq k$.}\end{cases}
    \]
    Furthermore, we have the following key theorem known as ``sign-coherence of $c$-vectors''.
    \begin{theorem}  
      \label{th:sign coherence}
      \cite[Corollary 5.5]{gross2018canonical}
      \cite[Theorem 3.20]{nakanishi2015structure}
      Each c-vector \(\cc_{j;t} = (c_{ij;t})_{i=1}^n\) is non-zero and its components are either all non-negative or all non-positive.
    \end{theorem}
    Accordingly, we set the tropical sign \(\varepsilon_{j;t}\) as \(+1\) in the former case and \(-1\) in the latter case.
    Still working inside the universal semifield \(\QQ_+(\yy,\zz)\), we make the following definition.
    \begin{definition}
      For \(t \in \TT_n\) and \(1 \leq i \leq n\), we define an \emph{F-polynomial} \(F_{i;t}(\yy,\zz) \in \QQ_+(\yy,\zz)\), using the following recursion relations:
      \[
        F_{i;t_0}(\yy,\zz) = 1
      \]
      and, when \(t,t' \in \TT_n\) are joined by an edge labeled \(k\), we have
      \[
        F_{i,t'}(\yy,\zz) =
        \begin{cases}
          F_{k,t}(\yy,\zz)^{-1} \left(\prod\limits_{j=1}^n y_{j;t_0}^{[-c_{jk;t} d_k]_+} F_{j;t}(\yy,\zz)^{[-b_{jk;t} d_k]_+}\right) P_k\left(\prod\limits_{j=1}^n y_{j;t_0}^{c_{jk;t}} F_{j;t}(\yy,\zz)^{b_{jk;t}}\right) & \text{if $i=k$;} \\
          F_{i;t}(\yy,\zz) & \text{if $i \neq k$.}
        \end{cases}
      \]
    \end{definition}
    Following \cite[Proposition 3.12]{nakanishi2015structure}, as the name suggests, the $F$-polynomials are actually elements of $\ZZ[\yy,\zz]$.
    As a consequence of Theorem~\ref{th:sign coherence} and \cite[Proposition 3.19]{nakanishi2015structure} (c.f. \cite[Proposition 5.6]{fomin2007cluster}), we have the following.
    \begin{theorem}
      \label{th:constant term}
      Each $F$-polynomial $F_{i;t}(\yy,\zz)$ has constant term 1 when considered as a polynomial in $\yy$.
    \end{theorem}
    
    Our main use of the $F$-polynomials is the following ``separation of additions" formula for the \(y\)-variables.
    Since the $F$-polynomials can be viewed as elements of $\QQ_+(\yy,\zz)$, we may specialize the variables $\yy$ and $\zz$ in an arbitrary semifield, see Example~\ref{ex:universal semifield}.
    Thus given an initial seed $(B_{t_0},\yy_{t_0},\zz_{t_0})$ over a semifield $\PP$, we set $F_{i;t}|_\PP(\yy_{t_0},\zz_{t_0})$ to be this specialization.
    Then we have the following.
    \begin{theorem}
      \cite[Theorem 3.22]{nakanishi2015structure}
      Fix an initial seed $(B_{t_0},\yy_{t_0},\zz_{t_0})$ over a semifield $\PP$.
      For \(t \in \TT_n\) and \(1 \leq j \leq n\), we have
      \begin{equation}
        \label{eq:separation of additions}
        y_{j;t} = \prod\limits_{i=1}^n y_{i;t_0}^{c_{ij;t}} \big(F_{i;t}|_\PP(\yy_{t_0},\zz_{t_0})\big)^{b_{ij;t}}.
      \end{equation}
    \end{theorem}
    \begin{definition}
      \label{def:periodicity}
      Given a sequence of seed mutations
      \[\big(B[1],\yy[1],\zz[1]\big) \stackrel{\mu_{k_1}}{\joinrel\relbar\!\longrightarrow} \cdots \stackrel{\mu_{k_M}}{\joinrel\relbar\!\longrightarrow} \big(B[M+1],\yy[M+1],\zz[M+1]\big),\] we say that the sequence is \emph{\(\sigma\)-periodic} for some \(n\)-permutation \(\sigma\) if
      \begin{itemize}
        \item $b_{\sigma(i)\sigma(j)}[M+1] = b_{ij}[1]$ for $1 \leq i,j \leq n$;
        \item $y_{\sigma(i)}[M+1] = y_{i}[1]$ for $1 \leq i \leq n$;
        \item $z_{\sigma(i),s}[M+1] = z_{i,s}[1]$ for $1 \leq i \leq n$ and $1 \leq s \leq d_i-1$.
      \end{itemize}
    \end{definition}
  
    Via specialization, any $\sigma$-periodicity of seeds over $\QQ_+(\yy,\zz)$ starting with initial exchange matrix $B[1]$ gives rise to a $\sigma$-periodicity over any semifield starting with the same initial exchange matrix $B[1]$.
    We call such periodicities \emph{proper} to distinguish from periodicities that arise due to special properties of $\PP$ or of the choice of initial $y$- or $z$-variables.
  
    As we will see in Section~\ref{sec:dilogarithm identities}, proper \(\sigma\)-periodicities can be used to construct identities of Rogers dilogarithms.
    The following observation will be valuable for our reasoning.
    \begin{prop}
      \label{prop:z periodicity}
      Given any proper periodicity of seeds as in Definition~\ref{def:periodicity} and any $1\le i\le n$ with $d_i>1$, we have $\sigma(i)=i$ and there can only exist an even number of mutation directions $k_\ell$ with $k_\ell=i$.
    \end{prop}
    \begin{proof}
      Consider a periodicity of seed mutations
      \[\big(B[1],\yy[1],\zz[1]\big) \stackrel{\mu_{k_1}}{\joinrel\relbar\!\longrightarrow} \cdots \stackrel{\mu_{k_M}}{\joinrel\relbar\!\longrightarrow} \big(B[M+1],\yy[M+1],\zz[M+1]\big)\]
      over $\QQ_+(\yy,\zz)$ and suppose there is some \(1 \leq i \leq n\) such that \(d_i > 1\).
      We have \(z_{i,s}[M+1] = z_{\sigma(i),s}[1]\) for all \(1 \leq s \leq d_i-1\) so we must have \(d_i = d_{\sigma(i)}\).
      However, by definition of the $z$-variable mutations, we must also have \(z_{i,s}[M+1] = z_{i,s^\circ}[1]\), where  \(s^\circ = s\) if there are an even number of \(1 \leq \ell \leq M\) where \(k_\ell =i\), and \(s^\circ = d_i - s\) if this number is odd.
      But then we have \(z_{\sigma(i),s}[1] = z_{i,s^\circ}[1]\) for all \(1 \leq s \leq d_i-1\) and this can happen only if \(\sigma(i)=i\) and \(s^\circ = s\), i.e. there are an even number of $\ell$ with $k_\ell=i$.
    \end{proof}
  
    \begin{definition}
      Given any seed $(B,\yy,\zz)$ over $\PP$, define its \emph{right companion seed} $(BD,{}^{\mathcal{R}}\yy)$, where ${}^{\mathcal{R}}\yy \coloneq (y_1^{d_1},\ldots,y_n^{d_n})$ and $D=\diag(d_1,\ldots,d_n)$.
    \end{definition}
  
    The following is a key result of \cite{nakanishi2015companion}.
    \begin{theorem}
      \label{th:companion periodicity}
      For any $t\in\TT_n$, the right companion seed of $(B_t,\yy_t,\zz_t)$ is obtained from the right companion $(B_{t_0}D,{}^{\mathcal{R}}\yy_{t_0})$ of $(B_{t_0},\yy_{t_0},\zz_{t_0})$ by the sequence of mutations from $t_0$ to $t$.
      In particular, any proper periodicity of seeds
      \[\big(B[1],\yy[1],\zz[1]\big) \stackrel{\mu_{k_1}}{\joinrel\relbar\!\longrightarrow} \cdots \stackrel{\mu_{k_M}}{\joinrel\relbar\!\longrightarrow} \big(B[M+1],\yy[M+1],\zz[M+1]\big)\]
      gives rise to a proper periodicity of companions
      \[\big(B[1]D,{}^{\mathcal{R}}\yy[1]\big) \stackrel{\mu_{k_1}}{\joinrel\relbar\!\longrightarrow} \cdots \stackrel{\mu_{k_M}}{\joinrel\relbar\!\longrightarrow} \big(B[M+1]D,{}^{\mathcal{R}}\yy[M+1]\big).\]
    \end{theorem}

  \subsection{Cluster Groupoids} 
  \label{sec:groupoids}
    One way to study cluster algebras is to introduce their geometric counterparts -- the cluster manifolds -- and endow them with additional geometric structure, namely a compatible Poisson bracket.
    For this approach, we consider below only $\PP=\RR_{>0}$.
    Then associated to each seed is an affine space and the cluster mutations become Poisson maps between these affine spaces.
    The cluster manifold is then the gluing along the mutation maps of these affine charts.
    
    In \cite{li2020symplectic}, Lie groupoids over the cluster manifolds integrating the compatible Poisson structure were constructed.
    Here, mutations of cluster charts (as Poisson maps) were lifted to morphisms between symplectic groupoid charts.
    The cluster groupoid is the result of gluing these charts.
    Crucially, a periodicity of cluster algebra mutations lifts to a periodicity of groupoid mutations.
  
    An important aspect of the perspective taken in \cite{li2020symplectic} is incorporating variability of the $z$-variables into this geometric approach.
    This necessitates the introduction of groupoid fiber coordinates, designated as $a$-variables, over the affine space associated to the $z$-variables.
    We will show that the periodicity of the $a$-variables is directly related to dilogarithm identities.

    As above we are really only interested in the combinatorial aspects of the cluster groupoid, so we only discuss their seeds and mutations.
    The only part of the groupoid structure we will utilize is the action map and so we omit any discussion of the remaining structures.

    \begin{definition}
      A \emph{groupoid $y$-seed}, or simply a \emph{groupoid seed}, is a quintuple $(B,\yy,\zz,\qq,\aa)$, where:
      \begin{itemize}
        \item $(B,\yy,\zz)$ is a seed in $\RR_{>0}$;
        \item $\qq=(q_i)_{i=1}^n$ is an $n$-tuple of elements of $\RR$, called the $q$-variables;
        \item $\aa=(a_i)_{i=1}^n$ is an $n$-tuple of elements of $\RR$, called the $a$-variables.
      \end{itemize}
      The \emph{action} defines a seed $\beta(B,\yy,\zz,\qq,\aa)=(B,\beta\yy,\zz)$, where $\beta(y_j):=y_j e^{\sum_{i} r_i b_{ij} y_i q_i}$ which we can naturally extend to any expressions involving $\yy$.
      For any integer $1\le k\le n$, the \emph{mutation} of $(B,\yy,\zz,\qq,\aa)$ with sign $\varepsilon=\pm1$ in direction $k$ is a new groupoid seed $(B',\yy',\zz',\qq',\aa')$, where $(B',\yy',\zz')$ is the mutation of $(B,\yy,\zz)$ in direction $k$ and
      \begin{align*}
        q_i' &\coloneq
        \begin{cases} 
          -q_k y_k^2+ \sum\limits_{i'=1}^n [\varepsilon d_k b_{ki'}]_+ q_{i'} y_{i'} y_k+\frac{y_k}{r_k} \log\!\left(\frac{P^\circ_k(\beta(y_k^\varepsilon))}{P^\circ_k(y_k^\varepsilon)}\right) & \text{if $i=k$;}\\
          q_i y_k^{-[\varepsilon d_k b_{ki}]_+} \big(P^\circ_k(y_k^\varepsilon)\big)^{b_{ki}} & \text{if $i \neq k$;}
        \end{cases}\\
        a_{i,s}' &\coloneq
        \begin{cases}
          a_{k,d_k-s} + \frac{\varepsilon}{r_k} \int_{y_k^\varepsilon}^{\beta(y_k^\varepsilon)} \frac{u^{\varepsilon (d_k-s) -1}}{P_k(u^\varepsilon)} \delta u & \text{if $i=k$;}\\
          a_{i,s} & \text{if $i \neq k$;}
        \end{cases}
      \end{align*}
      where \(P_k^\circ(\alpha) \coloneq P_k(\alpha^\varepsilon) \alpha^{(\frac{1-\varepsilon}{2})d_k}\).
    \end{definition}
    
    \begin{prop}
      \label{prop:compatibility}
      The action commutes with mutations.
    \end{prop}
    \begin{proof}
      Suppose $(B,\yy,\zz,\qq,\aa)$ and $(B',\yy',\zz',\qq',\aa')$ are related by mutation of direction $k$.
      In order to distinguish this from the extended action of $\beta$ on expressions in $\yy$, we write $\beta(B',\yy',\zz',\qq',\aa')=(B',\beta'\yy',\zz')$.
      
      To start, we observe that
      \[
        y'_i q'_i
        = 
        \begin{cases} 
          -q_k y_k + \sum\limits_{i'=1}^n [\varepsilon d_k b_{ki'}]_+ q_{i'} y_{i'} + \frac{1}{r_k} \log\!\left(\frac{P^\circ_k(\beta(y_k^\varepsilon))}{P^\circ_k(y_k^\varepsilon)}\right) & \text{if $i=k$;}\\
          y_i q_i & \text{if $i \neq k$.}
        \end{cases}
      \]
      For \(i\neq k\), we have
      \begin{align*}
        \beta'(y'_i)
        &= y_i' e^{\sum_j r_j b'_{ji} y'_j q'_j}
        = y_i' e^{\sum_{j \neq k} r_j (b_{ji}+[-\varepsilon b_{jk} d_k]_+ b_{ki} + b_{jk}[\varepsilon d_k b_{ki}]_+) y'_j q'_j} \cdot e^{-r_k b_{ki} y'_k q'_k}\\
        &= y_i y_k^{[\varepsilon d_k b_{ki}]_+} P^\circ_k(y_k^\varepsilon)^{-b_{ki}} e^{\sum_{j \neq k} r_j (b_{ji}+[-\varepsilon b_{jk} d_k]_+ b_{ki} + b_{jk}[\varepsilon d_k b_{ki}]_+) y_j q_j}\\
        &\hspace{2in} \cdot e^{-r_k b_{ki} (-q_k y_k + \sum_{j} [\varepsilon d_k b_{kj}]_+ y_j q_j)} \left(\frac{P^\circ_k(\beta(y_k^\varepsilon))}{P^\circ_k(y_k^\varepsilon)}\right)^{-b_{ki}}\\ 
        &= y_i e^{\sum_{j} r_j b_{ji} y_j q_j} y_k^{[\varepsilon d_k b_{ki}]_+} \!\left(e^{\sum_{j} r_j b_{jk} y_j q_j}\right)^{[\varepsilon d_k b_{ki}]_+} P^\circ_k(\beta(y_k^\varepsilon))^{-b_{ki}} \\
        &= \beta\!\left(y_i y_k^{[\varepsilon d_k b_{ki}]_+} P^\circ_k(y_k^\varepsilon)^{-b_{ki}}\right)=\beta(y'_i),
      \end{align*}
      where the third equality used that $-r_j b_{jk}=r_k b_{kj}$.
      We also have 
      \begin{align*}
        \beta'(y'_k) 
        = y_k' e^{\sum_j r_j b'_{jk} y'_j q'_j}
        = y_k^{-1} e^{\sum_j -r_j b_{jk} y_j q_j}
        =\left(y_k e^{\sum_j r_j b_{jk} y_j q_j}\right)^{-1}=\beta(y_k^{-1})=\beta(y'_k).
      \end{align*}
      Altogether we see that the action gives the same result whether it is applied before or after mutations.
    \end{proof}
  
    As with $y$-seeds, we associate groupoid seeds to the $n$-regular tree $\TT_n$ as follows.
    Assign a groupoid seed $(B_{t_0},\yy_{t_0},\zz_{t_0},\qq_{t_0},\aa_{t_0})$ to the root vertex $t_0\in\TT_n$.
    Then for $t,t'\in\TT_n$ which are joined by an edge labeled $k$, the groupoid seed $(B_{t'},\yy_{t'},\zz_{t'},\qq_{t'},\aa_{t'})$ is obtained from the groupoid seed $(B_t,\yy_t,\zz_t,\qq_t,\aa_t)$ by the mutation in direction $k$ with sign $\varepsilon=\varepsilon_{k;t}$, the tropical sign at this seed.
  
    \begin{prop}
      \label{prop:groupoid periodicity}
      \cite[Proposition 4.25]{li2020symplectic}
      Given any proper $\sigma$-periodicity of seeds 
      \[\big(B[1],\yy[1],\zz[1]\big) \stackrel{\mu_{k_1}}{\joinrel\relbar\!\longrightarrow} \cdots \stackrel{\mu_{k_M}}{\joinrel\relbar\!\longrightarrow} \big(B[M+1],\yy[M+1],\zz[M+1]\big)\]
      and an extension of $\big(B[1],\yy[1],\zz[1]\big)$ to a groupoid seed $\big(B[1],\yy[1],\zz[1],\qq[1],\aa[1]\big)$, the corresponding sequence of groupoid seed mutations
      \[\big(B[1],\yy[1],\zz[1],\qq[1],\aa[1]\big) \stackrel{\mu_{k_1}}{\joinrel\relbar\!\longrightarrow} \cdots \stackrel{\mu_{k_M}}{\joinrel\relbar\!\longrightarrow} \big(B[M+1],\yy[M+1],\zz[M+1],\qq[M+1],\aa[M+1]\big)\]
      is also $\sigma$-periodic.
    \end{prop}
  
    A $\sigma$-periodicity as in Proposition~\ref{prop:groupoid periodicity} is called \emph{proper}.
    Everywhere below, we take $\big(B[\ell],\yy[\ell],\zz[\ell],\qq[\ell],\aa[\ell]\big)=\big(B_{t_\ell},\yy_{t_\ell},\zz_{t_\ell},\qq_{t_\ell},\aa_{t_\ell}\big)$ for $t_\ell\in\TT_n$, then define $\varepsilon_\ell:=\varepsilon_{k_\ell;t_\ell}$ and $P_\ell:=P_{k_\ell;t_\ell}$.
    The following is an immediate consequence of Proposition~\ref{prop:z periodicity}.
    \begin{cor}
      \label{cor:a variable periodicity}
      For any proper $\sigma$-periodicity of groupoid seeds
      \[\big(B[1],\yy[1],\zz[1],\qq[1],\aa[1]\big) \stackrel{\mu_{k_1}}{\joinrel\relbar\!\longrightarrow} \cdots \stackrel{\mu_{k_M}}{\joinrel\relbar\!\longrightarrow} \big(B[M+1],\yy[M+1],\zz[M+1],\qq[M+1],\aa[M+1]\big),\]
      we have $a_{i,s}[1]=a_{i,s}[M+1]$ for $1\le i\le n$ and $1\le s\le d_i-1$.
    \end{cor}
  
    Using this and the periodicity of the $a$-variables from Proposition~\ref{prop:groupoid periodicity}, we obtain the following result which is crucial to our approach to dilogarithm identities.
    \begin{theorem} 
      \label{th:thm from periodicity}
      Given a proper $\sigma$-periodicity
      \[\big(B[1],\yy[1],\zz[1],\qq[1],\aa[1]\big) \stackrel{\mu_{k_1}}{\joinrel\relbar\!\longrightarrow} \cdots \stackrel{\mu_{k_M}}{\joinrel\relbar\!\longrightarrow} \big(B[M+1],\yy[M+1],\zz[M+1],\qq[M+1],\aa[M+1]\big)\]
      of groupoid seeds, for \(1 \leq j \leq n\) and \(1 \leq s \leq d_j-1\), we have
      \begin{equation}
        \label{eq:a variable periodicity}
        \sum\limits_{\ell:k_\ell = j} \frac{\varepsilon_\ell}{r_j} \int_{{y_j[\ell]}^{\varepsilon_{\ell}}}^{\beta(y_{j}[\ell])^{\varepsilon_{\ell}}} \frac{u^{\varepsilon_{\ell} s_\ell^\circ-1}}{P_{\ell}(u^{\varepsilon_{\ell}})} \ \delta u = 0,
      \end{equation}
      where \(s_\ell^\circ = s\) if there are an even number of \(1 \leq \ell' \leq \ell-1\) with \(k_{\ell'} =j\) and \(s_\ell^\circ = d_j - s\) if this number is odd.
    \end{theorem}
      
    We now derive analogous identities which will be more useful for our purposes under the assumption that the exchange matrix \(B\) is full rank.
    The following technical observation is key to the proof.
    \begin{lemma}
      \label{le:careful choice of q}
      Consider a seed $(B,\yy,\zz)$ for which $B$ is full rank.
      Then, for any $\lambda\in\RR_{>0}$, there exists a choice of $\qq$ so that $\sum_i r_i b_{ij} y_i q_i=-\lambda$ for all $1\le j\le n$.
    \end{lemma}
    \begin{proof}
      To see this, we collect the expressions into a vector and observe the following factorization:
      \[
        \begin{bmatrix}
          -\lambda \\
          \vdots \\
          -\lambda
        \end{bmatrix}
        =
        \begin{bmatrix}
          \sum_{i} r_i b_{i1} y_{i} q_{i} \\
          \vdots \\
          \sum_{i} r_i b_{in} y_{i} q_{i}
        \end{bmatrix}
        =
        (RB)^T
        \begin{bmatrix}
          y_{1} q_{1} \\
          \vdots \\
          y_{n} q_{n}
        \end{bmatrix}
        =
        -RB\diag(\yy)
        \begin{bmatrix}
          q_{1} \\
          \vdots \\
          q_{n}
        \end{bmatrix}.
      \]
      Since $B$ is full rank, $-RB$ is invertible and, since $y_i\ne 0$ for all $i$, $-RB\diag(\yy)$ is also invertible.
      The result follows.
    \end{proof}

    \begin{theorem}
      \label{th:analogues}
      Given a proper $\sigma$-periodicity
      \[\big(B[1],\yy[1],\zz[1],\qq[1],\aa[1]\big) \stackrel{\mu_{k_1}}{\joinrel\relbar\!\longrightarrow} \cdots \stackrel{\mu_{k_M}}{\joinrel\relbar\!\longrightarrow} \big(B[M+1],\yy[M+1],\zz[M+1],\qq[M+1],\aa[M+1]\big)\]
      of groupoid seeds with $B[1]$ of full rank, for \(1 \leq j \leq n\) and \(1 \leq s \leq d_j-1\), we have
      \begin{equation}
        \label{eq:analogues}
        \sum\limits_{\ell:k_\ell = j} \frac{\varepsilon_{\ell}}{r_j} \int_{0}^{y_{j}[\ell]^{\varepsilon_{\ell}}} \frac{u^{\varepsilon_{\ell} s_\ell^\circ-1}}{P_{\ell}(u^{\varepsilon_{\ell}})} \ \delta u = 0,
      \end{equation}
      where \(s_\ell^\circ = s\) if there are an even number of \(1 \leq \ell' \leq \ell-1\) with \(k_{\ell'} =j\) and \(s_\ell^\circ = d_j - s\) if this number is odd.
    \end{theorem}
    \begin{proof}
      The result of Theorem~\ref{th:thm from periodicity} can be rewritten as
      \begin{equation}
        \label{eq:initial}
        \sum\limits_{\ell:k_\ell = j} \frac{\varepsilon_{\ell}}{r_j} \int_{0}^{\beta(y_{j}[\ell])^{\varepsilon_{l}}} \frac{u^{\varepsilon_{\ell} s_\ell^\circ-1}}{P_\ell(u^{\varepsilon_{\ell}})} \ \delta u
        =
        \sum\limits_{\ell:k_\ell = j} \frac{\varepsilon_{\ell}}{r_j} \int_{0}^{y_{j}[\ell]^{\varepsilon_{\ell}}} \frac{u^{\varepsilon_{\ell} s_\ell^\circ-1}}{P_\ell(u^{\varepsilon_{\ell}})} \ \delta u.
      \end{equation}
      While the left hand side of \eqref{eq:initial} apparently depends on the initial \(q\)-variables, the right hand side shows that it does not.

      Since the $\sigma$-periodicity is proper, by Proposition~\ref{prop:groupoid periodicity}, \eqref{eq:a variable periodicity} holds for any choice of $\qq[1]$ and hence any choice of $\qq_{t_0}$.
      By Lemma~\ref{le:careful choice of q}, we may choose $\qq_{t_0}$ so that each $\beta(y_{j;t_0})$ is arbitrarily small.

      By the separation of additions formula \eqref{eq:separation of additions} and Proposition~\ref{prop:compatibility}, we may write 
      \[
        \beta(y_{j;t})^{\varepsilon_{j;t}}
        =
        \prod\limits_{i=1}^n \beta(y_{i;t_0})^{\varepsilon_{j;t} c_{ij;t}} F_{i;t}|_{\RR_{>0}}(\beta(\yy_{t_0}),\zz_{t_0})^{\varepsilon_{j;t} b_{ij;t}}
      \]
      for any \(t \in \TT_n\) and \(1 \leq j \leq n\).
      By Theorem~\ref{th:constant term}, $F_{i;t}(\yy,\zz)$ has constant term $1$ as a polynomial in $\yy$ and thus making each $\beta(y_{j;t_0})$ arbitrarily small also makes $F_{i;t}|_{\RR_{>0}}(\beta(\yy_{t_0}),\zz_{t_0})$ arbitrarily close to $1$.
      Now observe that each exponent $\varepsilon_{j;t} c_{ij;t}$ is non-negative with at least one strictly positive by Theorem~\ref{th:sign coherence}.
      In particular, this implies $\prod\limits_{i=1}^n \beta(y_{i;t_0})^{\varepsilon_{j;t} c_{ij;t}}$ can be made arbitrarily small as the $\beta(y_{i;t_0})$ are all made arbitrarily small.
      
      It follows that the left hand side of \eqref{eq:initial} can be made arbitrarily small and thus the right hand side must be zero.
    \end{proof}

\section{Dilogarithm identities} 
  \label{sec:dilogarithm identities}
  
  \subsection{Dilogarithm definitions}
    The \emph{Euler dilogarithm} is given by
    \[
      \Li_2(x) \coloneq \sum\limits_{n=1}^\infty \frac{x^n}{n^2} = -\int_{0}^{-x} \frac{\log(1+y)}{y} \ \delta y \quad (x < 1)
    \]
    with corresponding \emph{Rogers dilogarithm} given by
    \begin{align*}
      \LL(x) &\coloneq \frac{1}{2} \int_{0}^x \left(\frac{\log\!\big(1+y\big)}{y}-\frac{\log(y)}{1+y} \right) \ \delta y\\
      &= -\Li_2(-x) -\frac{1}{2} \log(x) \log(1+x) \quad (x > 0).
    \end{align*}
    
    We generalize these as follows.
    Let \(P(x)\) be a monic polynomial of degree \(d\) with non-negative coefficients and $P(0)=1$.
    Then following \cite{nakanishi2015quantum}, define the \emph{Euler dilogarithm of higher degree} with respect to \(P(x)\) as
    \[
      \Li_{2,P}(x) \coloneq -\int_{0}^{-x} \frac{\log\!\big(P(y)\big)}{y} \ \delta y \quad (x \leq 0).
    \]
    We are most concerned with the \emph{Rogers dilogarithm of higher degree}, defined in \cite{nakanishi2018rogers} as:
    \begin{align}
      \label{eq:roger definition higher degree}
      \LL_{P}(x) &\coloneq \frac{1}{2} \int_{0}^x \left(\frac{\log\!\big(P(y)\big)}{y}-\frac{\log(y)}{P(y)} P'(y) \right) \ \delta y\\
      \label{eq:roger to euler} &= -\Li_{2,P}(-x) -\frac{1}{2} \log(x) \log\!\big(P(x)\big) \quad (x > 0).
    \end{align}
    These dilogarithms satisfy a remarkable variety of identities.
    One particularly easy identity to verify (see \cite{nakanishi2018rogers}) is the following:
    \begin{equation}
      \label{eq:constancy example}
      \LL_P(x)+\LL_{P^*}(x^{-1}) = \LL_P(\infty) = \LL_{P^*}(\infty),
    \end{equation}
    where \(P^*(x) \coloneq x^d P(x^{-1})\), i.e. \(P(x)\) with reversed coefficients, and \(\LL_P(\infty) \coloneq \lim\limits_{x \to \infty} \LL_P(x)\).
    
    This is not the only identity satisfied by Rogers dilogarithms.
    In fact, as we will see shortly, there are many identities of a very similar form.
    Next we introduce the \emph{constancy condition}, which is a powerful tool for constructing identities of Rogers dilogarithms.
    This will allow us to deduce identities of the dilogarithms of higher degree from the classical dilogarithm identities associated to a periodicity of cluster mutations.

    To begin this process, we make the following important observation.
    \begin{prop}
      \label{prop:zero coefficients}
      Suppose $P(x)=1+x^d$.
      Then
      \[
        \LL_P(x) = \frac{1}{d}\LL(x^d).
      \]
    \end{prop}
    \begin{proof}
      We apply the substitution $y=u^{1/d}$ and $\delta y=\frac{1}{d u^{(d-1)/d}} \delta u$ to get
      \begin{align*}
        \LL_P(x) 
        &= \int_{0}^x \frac{\log\!\big(1+y^d\big)}{y} \ \delta y - \frac{1}{2} \log(x) \log\!\big(P(x)\big) \\
        &= \frac{1}{d}\int_{0}^{x^d} \frac{\log\!\big(1+u\big)}{u} \ \delta u - \frac{1}{2d} \log(x^d) \log\!\big(1+x^d\big)\\
        &=\frac{1}{d} \LL(x^d).
      \end{align*}
    \end{proof}
    
    We will use this by applying \cite[Theorem 2.12]{nakanishi2018rogers} (c.f. \cite[Theorem 2.9]{kashaev2011classical}) in the following way.
    \begin{theorem}
      \label{th:companion dilog}
      Let
      \[\big(B[1],\yy[1],\zz[1]\big) \stackrel{\mu_{k_1}}{\joinrel\relbar\!\longrightarrow} \cdots \stackrel{\mu_{k_M}}{\joinrel\relbar\!\longrightarrow} \big(B[M+1],\yy[M+1],\zz[M+1]\big)\]
      be a proper $\sigma$-periodicity of seeds over $\RR_{>0}$ with corresponding proper periodicity
      \[\big(B[1]D,{}^{\mathcal{R}}\yy[1]\big) \stackrel{\mu_{k_1}}{\joinrel\relbar\!\longrightarrow} \cdots \stackrel{\mu_{k_M}}{\joinrel\relbar\!\longrightarrow} \big(B[M+1]D,{}^{\mathcal{R}}\yy[M+1]\big)\]
      of right companions.
      Then 
      \[\sum\limits_{\ell=1}^M \frac{\varepsilon_\ell}{d_{k_\ell}r_{k_\ell}} \LL\!\left(({}^{\mathcal{R}} y_{k_\ell}[\ell])^{\varepsilon_\ell}\right)=0.\]
    \end{theorem}
    \begin{proof}
      In order to apply \cite[Theorem 2.12]{nakanishi2018rogers} to the proper periodicity of right companions, it suffices to observe that the matrix $\diag(r_1d_1,\ldots,r_nd_n)$ is the common skew-symmetrizer of the exchange matrices $B[\ell]D$ appearing in the periodicity of right companions.
    \end{proof}

  \subsection{Constancy Condition}
    Let \(\mathcal{C}\) be the set of differentiable functions \(y(z):\RR_{>0} \to \RR_{>0}\).
    Regarding \(\mathcal{C}\) as a multiplicative abelian group, let \(\mathcal{C} \otimes \mathcal{C}\) be the tensor product of \(\mathcal{C}\) with itself, i.e.\ the additive abelian group generated by the formal, non-commutative products \(\{y_1 \otimes y_2 \ | \ y_1,y_2 \in \mathcal{C}\}\) with the following relations:
    \[
      (y_1y_2) \otimes y_3 = y_1 \otimes y_3+y_2 \otimes y_3
      \quad \text{and} \quad
      y_1 \otimes (y_2y_3) = y_1 \otimes y_2 + y_1 \otimes y_3.
    \]
    Below we will abuse notation slightly and write $\mathcal{C}\otimes\mathcal{C}$ for the rational vector space $\QQ\otimes_\ZZ(\mathcal{C}\otimes\mathcal{C})$.
    
    Define the \emph{wedge space} of \(\mathcal{C}\) as \(\bigwedge^2\mathcal{C} \coloneq (\mathcal{C} \otimes \mathcal{C}) / S^2 \mathcal{C}\), where \(S^2 \mathcal{C}\) is the subspace generated by elements of the form \(y_1 \otimes y_1 \in \mathcal{C} \otimes \mathcal{C}\).
    For any \(y_1,y_2 \in \mathcal{C}\), we define their \emph{wedge product} to be  \(y_1 \wedge y_2 \coloneq y_1 \otimes y_2 + S^2\mathcal{C} \in \bigwedge^2\mathcal{C}\).
    
    The following result, adapted from \cite{nakanishi2018rogers} will be the key step in establishing dilogarithm identities.
    \begin{lemma}
      \label{le:symmetry from wedge}
      Let \(y_1,\cdots,y_M,p_1,\cdots,p_M \in \mathcal{C}\) be functions such that \(\sum\limits_{\ell=1}^M c_\ell (y_\ell \wedge p_\ell) = 0\) for some $c_\ell\in\QQ$.
      Then \[\sum\limits_{\ell=1}^M c_\ell \frac{\delta}{\delta z}\Big(\log\!\big(y_\ell(z)\big)\Big) \log\!\big(p_\ell(z)\big) = \sum\limits_{\ell=1}^M c_\ell \log\!\big(y_\ell(z)\big) \frac{\delta}{\delta z}\Big(\log\!\big(p_\ell(z)\big)\Big).\]
    \end{lemma}
    \begin{proof}
      Since \(\sum\limits_{\ell=1}^M c_\ell (y_\ell \wedge p_\ell) = 0\) in \(\bigwedge^2\mathcal{C}\), we may rewrite it as
      \begin{equation}
        \label{eq:symmetric}
        \sum\limits_{\ell=1}^M c_\ell (y_\ell \otimes p_\ell) = \sum\limits_{i=1}^k h_i \otimes h_i,
      \end{equation}
      for some functions \(h_1,\cdots,h_k \in \mathcal{C}\).
      Next, we observe that for each \(z,w \in \RR_{>0}\), there exists a unique additive group homomorphism 
      \[\Phi_{z,w}: \mathcal{C}\otimes \mathcal{C} \to \RR,\qquad f \otimes g \mapsto \log\!\big(f(z)\big) \log\!\big(g(w)\big).\]
      Applying $\Phi_{z,w}$ and $\Phi_{w,z}$ to the right hand side of \eqref{eq:symmetric} yields the same result and thus
      \[\sum\limits_{\ell=1}^M c_\ell \log\!\big(y_\ell(z)\big) \log\!\big(p_\ell(w)\big) = \sum\limits_{\ell=1}^M c_\ell \log\!\big(y_\ell(w)\big) \log\!\big(p_\ell(z)\big).\]
      It follows that
      \[\left.\frac{\delta }{\delta z}\right|_{w=z} \sum\limits_{\ell=1}^M c_\ell \log\!\big(y_\ell(z)\big)\log\!\big(p_\ell(w)\big) = \left. \frac{\delta }{\delta z}\right|_{w=z}\sum\limits_{\ell=1}^M c_\ell \log\!\big(y_\ell(w)\big)\log\!\big(p_\ell(z)\big)\]
      which gives the claim.
    \end{proof}
    The following is an adaptation of \cite[Theorem 3.2]{nakanishi2018rogers}.
    \begin{prop}
      \label{prop:constancy condition}
      Let \(P_\ell(x)\) be a sequence of monic polynomials with $P_\ell(0)=1$.
      Assume the coefficient of $x^{s_\ell}$ in $P_\ell$ is a function $p_\ell$ of the variable $z$ with all other coefficients non-negative constants.
      Suppose \(y_1,...,y_M \in \mathcal{C}\) are functions satisfying
      \[\sum\limits_{\ell=1}^M c_\ell \Big(y_\ell(z) \wedge P_\ell\big(y_\ell(z)\big)\Big) = 0\]
      inside \(\bigwedge^2\mathcal{C}\) 
      for some $c_\ell\in\QQ$.
      Then 
      \[\frac{\delta}{\delta z}\sum\limits_{\ell=1}^M c_\ell \LL_{P_\ell}\big(y_\ell(z)\big)=\sum\limits_{\ell=1}^M c_\ell \frac{\delta p_\ell}{\delta z} \int_{0}^{y_\ell(z)} \frac{y^{s_\ell-1}}{P_\ell(y)} \ \delta y.\]
    \end{prop}
    \begin{proof}
      Using the expression \eqref{eq:roger to euler} of the Rogers dilogarithm we get
      \begin{align*}
        \frac{\delta}{\delta z}\sum\limits_{\ell=1}^M c_\ell \LL_{P_\ell}\big(y_\ell(z)\big)
        &=\sum\limits_{\ell=1}^M c_\ell \frac{\delta p_\ell}{\delta z} \int_{0}^{y_\ell(z)} \frac{y^{s_\ell-1}}{P_\ell(y)} \ \delta y\\
        &\qquad+\frac{1}{2}\sum\limits_{\ell=1}^M c_\ell \left(\frac{\delta}{\delta  z}\Big(\log\!\big(y_\ell(z)\big)\Big) \log\!\big(P_\ell(y_\ell(z))\big)\right.\\
        &\hspace{1in} \left.-\log\!\big(y_\ell(z)\big) \frac{\delta}{\delta z}\Big(\log\!\big(P_\ell(y_\ell(z))\big)\Big)\right).
      \end{align*}
      The final sum is zero by Lemma~\ref{le:symmetry from wedge} and the result follows.
    \end{proof}
    
  \subsection{Dilogarithm identities from $\sigma$-periodicities}
    It turns out that every proper periodicity of a generalized cluster algebra mutations over $\RR_{>0}$ gives rise to a dilogarithm identity.
    The key result is the following constancy condition established by Nakanishi.
    \begin{theorem}
      \label{th:cluster wedge}
      \cite[Theorem 4.4]{nakanishi2018rogers}
      Let 
      \[\big(B[1],\yy[1],\zz[1]\big) \stackrel{\mu_{k_1}}{\joinrel\relbar\!\longrightarrow} \cdots \stackrel{\mu_{k_M}}{\joinrel\relbar\!\longrightarrow} \big(B[M+1],\yy[M+1],\zz[M+1]\big)\]
      be a proper \(\sigma\)-periodicity of seeds.
      Then
      \[\sum\limits_{\ell=1}^M \frac{1}{r_{k_\ell}} \Big( y_{k_\ell}[\ell] \wedge P_\ell\big(y_{k_\ell}[\ell]\big) \Big) = 0.\]
    \end{theorem}

    Nakanishi utilized Theorem~\ref{th:cluster wedge} to generate identities of the Rogers dilogarithms of higher degree as in Theorem~\ref{th:dilog constancy w epsilon} below.
    We will utilize Theorem~\ref{th:cluster wedge} to provide an alternative proof by reducing to the dilogarithm identities of a proper $\sigma$-periodicity of standard cluster seeds.
    To begin, we establish a modified version of the constancy condition which will be more useful for us.
    \begin{lemma}
      \label{le:constancy}
      Let 
      \[\big(B[1],\yy[1],\zz[1]\big) \stackrel{\mu_{k_1}}{\joinrel\relbar\!\longrightarrow} \cdots \stackrel{\mu_{k_M}}{\joinrel\relbar\!\longrightarrow} \big(B[M+1],\yy[M+1],\zz[M+1]\big)\]
      be a proper \(\sigma\)-periodicity of seeds.
      We have the following constancy condition:
      \[\sum\limits_{\ell=1}^M \frac{\varepsilon_\ell}{r_{k_\ell}} \Big(y_{k_\ell}[\ell]^{\varepsilon_\ell} \wedge P^\circ_\ell \big(y_{k_\ell}[\ell]^{\varepsilon_\ell}\big)\Big) = 0,\]
      where \(P_\ell^\circ(\alpha) \coloneq P_\ell(\alpha^{\varepsilon_\ell}) \alpha^{(\frac{1-\varepsilon_\ell}{2})d_{k_\ell}}\).
    \end{lemma}
    \begin{proof}
      This follows by a direct calculation from Theorem~\ref{th:cluster wedge}:
      \begin{align*}
        \sum\limits_{\ell=1}^M \frac{\varepsilon_\ell}{r_{k_\ell}} \Big(y_{k_\ell}[\ell]^{\varepsilon_\ell} \wedge P^\circ_\ell \big(y_{k_\ell}[\ell]^{\varepsilon_\ell}\big)\Big)
        &=\sum\limits_{\ell=1}^M \frac{\varepsilon_\ell}{r_{k_\ell}} \Big(y_{k_\ell}[\ell]^{\varepsilon_\ell} \wedge P_{\ell} \big(y_{k_\ell}[\ell]\big) y_{k_\ell}[\ell]^{(\frac{\varepsilon_\ell-1}{2})d_{k_\ell}}\Big)\\
        &=\sum\limits_{\ell=1}^M \frac{\varepsilon_\ell}{r_{k_\ell}} \Big(y_{k_\ell}[\ell]^{\varepsilon_\ell} \wedge P_\ell \big(y_{k_\ell}[\ell]\big)\Big)+ \frac{\varepsilon_\ell}{r_{k_\ell}} \Big(y_{k_\ell}[\ell]^{\varepsilon_\ell} \wedge y_{k_\ell}[\ell]^{(\frac{\varepsilon_\ell-1}{2})d_{k_\ell}}\Big)\\
        &=\sum\limits_{\ell=1}^M \frac{\varepsilon_\ell^2}{r_{k_\ell}} \Big(y_{k_\ell}[\ell] \wedge P_{\ell} \big(y_{k_\ell}[\ell]\big)\Big) = \sum\limits_{\ell=1}^M \frac{1}{r_{k_\ell}} \Big(y_{k_\ell}[\ell] \wedge P_{\ell} \big(y_{k_\ell}[\ell]\big)\Big) = 0
      \end{align*}
      as desired.
    \end{proof}
    \begin{theorem}
      \label{th:dilog constancy w epsilon}
      \cite[Theorem 4.8]{nakanishi2018rogers}
      Let 
      \[\big(B[1],\yy[1],\zz[1]\big) \stackrel{\mu_{k_1}}{\joinrel\relbar\!\longrightarrow} \cdots \stackrel{\mu_{k_M}}{\joinrel\relbar\!\longrightarrow} \big(B[M+1],\yy[M+1],\zz[M+1]\big)\]
      be a proper \(\sigma\)-periodicity of seeds with $B[1]$ of full rank.
      Then the following identity holds:
      \begin{equation}
        \label{eq:dilog constancy w epsilon}
        \sum\limits_{\ell=1}^M \frac{\varepsilon_\ell}{r_{k_\ell}} \LL_{P^\circ_\ell}\big((y_{k_\ell}[\ell])^{\varepsilon_\ell}\big)=0,
    \end{equation}
    where \(P_\ell^\circ(\alpha) \coloneq P_\ell(\alpha^{\varepsilon_\ell}) \alpha^{(\frac{1-\varepsilon_\ell}{2})d_{k_\ell}}\).
    \end{theorem}
    \begin{proof}
      We first observe that, for \(1 \leq j \leq n\) and \(1 \leq s \leq d_i-1\), we have
      \[
        \frac{\delta}{\delta z_{j,s}} \sum\limits_{\ell=1}^M \frac{\varepsilon_\ell}{r_{k_\ell}} \LL_{P^\circ_\ell}\big((y_{k_\ell}[\ell])^{\varepsilon_\ell}\big) = 0.
      \]
      Indeed, following Lemma~\ref{le:constancy} and Proposition~\ref{prop:constancy condition}, we have
      \[
        \frac{\delta}{\delta z_{j,s}} \sum\limits_{\ell=1}^M \frac{\varepsilon_\ell}{r_{k_\ell}} \LL_{P^\circ_\ell}\big((y_{k_\ell}[\ell])^{\varepsilon_\ell}\big) = 
        \sum\limits_{\ell:k_\ell = j} \frac{\varepsilon_{\ell}}{r_j} \int_{0}^{y_{j}[\ell]^{\varepsilon_{\ell}}} \frac{u^{\varepsilon_{\ell} s_\ell^\circ-1}}{P_{\ell}(u^{\varepsilon_{\ell}})} \ \delta u
      \]
      which is zero by Theorem~\ref{th:analogues}.
  
      Thus the expression 
      \[\sum\limits_{\ell=1}^M \frac{\varepsilon_\ell}{r_{k_\ell}} \LL_{P^\circ_\ell}\big((y_{k_\ell}[\ell])^{\varepsilon_\ell}\big)\]
      does not depend on the \(z\)-variables.
      Thus taking the $z$-variables to be zero and applying Proposition~\ref{prop:zero coefficients}, we get
      \[\sum\limits_{\ell=1}^M \frac{\varepsilon_\ell}{r_{k_\ell}} \LL_{P^\circ_\ell}\big((y_{k_\ell}[\ell])^{\varepsilon_\ell}\big)
      =
      \sum\limits_{\ell=1}^M \frac{\varepsilon_\ell}{d_{k_\ell}r_{k_\ell}} \LL\!\left((y_{k_\ell}[\ell])^{\varepsilon_\ell d_{k_\ell}}\right)
      =
      \sum\limits_{\ell=1}^M \frac{\varepsilon_\ell}{d_{k_\ell}r_{k_\ell}} \LL\!\left(({}^{\mathcal{R}} y_{k_\ell}[\ell])^{\varepsilon_\ell}\right)\]
      which is zero by Theorem~\ref{th:companion dilog}.
    \end{proof}
  
    \begin{remark}
      If instead we take Theorem~\ref{th:dilog constancy w epsilon} without the full rank assumption as proved in \cite{nakanishi2018rogers} as the starting point, we may reverse the logic applied in the proof above to show that Theorem~\ref{th:analogues} holds regardless of the rank of the exchange matrix \(B\).
    \end{remark}

    We can leverage Theorem~\ref{th:analogues} together with the original constancy condition Theorem~\ref{th:dilog constancy no epsilon} to prove an additional result about dilogarithm identities from periodicities.
    \begin{theorem}
      Let 
      \[\big(B[1],\yy[1],\zz[1]\big) \stackrel{\mu_{k_1}}{\joinrel\relbar\!\longrightarrow} \cdots \stackrel{\mu_{k_M}}{\joinrel\relbar\!\longrightarrow} \big(B[M+1],\yy[M+1],\zz[M+1]\big)\]
      be a proper \(\sigma\)-periodicity of seeds with $B[1]$ of full rank.
      For any \(1 \leq j \leq n\) and \(1 \leq s \leq d_i-1\), we have
      \[\frac{\delta}{\delta z_{j,s}} \sum\limits_{\ell=1}^M \frac{1}{r_{k_\ell}} \LL_{P_\ell}\big(y_{k_\ell}[\ell]\big) = \sum\limits_{\ell:k_\ell=j} \left(\frac{1-\varepsilon_\ell}{2r_j}\right) \int_0^\infty \frac{u^{s_\ell^\circ-1}}{P_\ell(u)} \ \delta u.\]
    \end{theorem}
    \begin{proof}
      From Proposition~\ref{prop:constancy condition}, we get
      \[\frac{\delta}{\delta z_{j,s}} \sum\limits_{\ell=1}^M \frac{1}{r_{k_\ell}} \LL_{P_\ell}\big(y_{k_\ell}[\ell]\big) = \sum\limits_{\ell:k_\ell=j} \frac{1}{r_j} \int_{0}^{y_{k_\ell}[\ell]}\frac{u^{s_\ell^\circ-1}}{P_\ell(u)} \ \delta u.\]
      From here we apply the change of variables \(v = u^{\varepsilon_\ell}\) in each term of \eqref{eq:analogues} to get
      \[
        \sum\limits_{\ell:k_\ell = j, \varepsilon_\ell=1} \frac{1}{r_j}\int_0^{y_{k_\ell}[\ell]} \frac{v^{s_\ell^\circ-1}}{P_\ell(v)} \ \delta v + \sum\limits_{\ell:k_\ell = j, \varepsilon_\ell=-1} \frac{1}{r_j}\int_{\infty}^{y_{k_\ell}[\ell]} \frac{v^{s_\ell^\circ-1}}{P_\ell(v)} \ \delta v
        = 0.
      \]
      This is equivalent to 
      \[
        \sum\limits_{\ell:k_\ell=j} \frac{1}{r_j} \int_{0}^{y_{k_\ell}[\ell]}\frac{v^{s_\ell^\circ-1}}{P_\ell(v)} \ \delta v = \sum\limits_{\ell:k_\ell=j} \left(\frac{1-\varepsilon_\ell}{2r_j}\right) \int_0^\infty \frac{v^{s_\ell^\circ-1}}{P_\ell(v)} \ \delta v
      \]
      which concludes the proof.
    \end{proof}
    
    This can be refined using Theorem~\ref{th:dilog constancy w epsilon} as follows.
    \begin{theorem}
      \label{th:dilog constancy no epsilon}
      \cite[Theorem 4.5]{nakanishi2018rogers}
      Let 
      \[\big(B[1],\yy[1],\zz[1]\big) \stackrel{\mu_{k_1}}{\joinrel\relbar\!\longrightarrow} \cdots \stackrel{\mu_{k_M}}{\joinrel\relbar\!\longrightarrow} \big(B[M+1],\yy[M+1],\zz[M+1]\big)\]
      be a proper \(\sigma\)-periodicity of seeds with $B[1]$ of full rank.
      Then 
      \[
        \sum\limits_{\ell=1}^M \frac{1}{r_{k_\ell}} \LL_{P_\ell}\big(y_{k_\ell}[\ell]\big) = \sum\limits_{\ell=1}^M \frac{1}{r_{k_\ell}} \left(\frac{1-\varepsilon_\ell}{2}\right) \LL_{P_\ell}(\infty),
      \]
      in particular it does not depend on the initial \(y\)-variables.
    \end{theorem}
    \begin{proof}
      We may rewrite \ref{eq:constancy example} as
      \[
        \LL_{P^\circ}(x^\varepsilon) = 
        \begin{cases}
        \LL_{P}(x) & \text{if $\varepsilon=1$;} \\
        \LL_{P}(\infty)-\LL_{P}(x) & \text{if $\varepsilon=-1$.}
        \end{cases}
      \]
      Substituting this into equation \eqref{eq:dilog constancy w epsilon} from Theorem~\ref{th:dilog constancy w epsilon} yields
      \[
        \sum\limits_{\ell = 1}^M \frac{1}{r_{k_\ell}} \LL_{P_\ell}\big(y_{k_\ell}[\ell]\big) - \sum\limits_{\ell ; \varepsilon_\ell=-1} \frac{1}{r_{k_\ell}} \LL_{P_\ell}(\infty) = 0
      \]
      which is equivalent to the claim.
    \end{proof}

  \bibliographystyle{amsalpha}
  \bibliography{bibliography.bib}
\end{document}